\documentclass[11pt,english]{article}

\usepackage[margin = 2cm, bottom=16 mm,footskip=7mm,top=18mm]{geometry}

\usepackage{amsthm}
\usepackage{amsmath}
\usepackage{amssymb}
\usepackage{setspace}
\usepackage{mathtools}
\usepackage{graphicx}
\usepackage[hidelinks]{hyperref}
\usepackage{cleveref}
\usepackage{hyperref}
\usepackage{enumitem}
\usepackage{framed}
\usepackage{subcaption}
\usepackage{xspace}
\usepackage[square,sort,comma,numbers]{natbib}
\usepackage{thmtools}
\usepackage{thm-restate}

\usepackage{floatrow}

\theoremstyle{plain}

\newtheorem*{thm*}{Theorem}
\newtheorem{thm}{Theorem}
\Crefname{thm}{Theorem}{Theorems}

\newtheorem*{lem*}{Lemma}
\newtheorem{lem}[thm]{Lemma}
\Crefname{lem}{Lemma}{Lemmas}

\newtheorem*{claim*}{Claim}

\crefname{claim}{Claim}{Claims}
\Crefname{claim}{Claim}{Claims}

\newtheorem{prop}[thm]{Proposition}
\Crefname{prop}{Proposition}{Propositions}

\crefname{cor}{Corollary}{Corollaries}

\crefname{conj}{Conjecture}{Conjectures}

\Crefname{qn}{Question}{Questions}

\Crefname{obs}{Observation}{Observations}

\Crefname{ex}{Example}{Examples}

\theoremstyle{definition}

\Crefname{prob}{Problem}{Problems}

\Crefname{defn}{Definition}{Definitions}

\theoremstyle{remark}

\renewenvironment{proof}[1][]{\begin{trivlist}
\item[\hspace{\labelsep}{\bf\noindent Proof#1.\/}] }{\qed\end{trivlist}}

\newcommand{\remove}[1]{}

\newcommand{\ceil}[1]{
    \left\lceil #1 \right\rceil
}
\newcommand{\floor}[1]{
    \left\lfloor #1 \right\rfloor
}

\newcommand{\lramsey}[2]{
\unskip$\left({#1},{#2}\right)$-locally Ramsey\xspace
}

\renewcommand{\P}{\mathbb{P}}
\newcommand{\G}{\mathcal{G}}

\title{\vspace{-1 cm}
Large cliques and independent sets all over the place}
\date{}
\author{Noga Alon\thanks{Department of Mathematics, Princeton University, Princeton, USA and Schools of Mathematics
and Computer Science, Tel Aviv University, Tel Aviv. Email: 
\href{mailto:nogaa@tau.ac.il} {\nolinkurl{nogaa@tau.ac.il}}. 
Research supported in part by NSF grant DMS-1855464, ISF grant 281/17, BSF grant 2018267 and the
Simons Foundation.}
\and
Matija Buci\'c\thanks{Department of Mathematics, ETH, Z\"urich, Switzerland. Email: \href{mailto:matija.bucic@math.ethz.ch} {\nolinkurl{matija.bucic@math.ethz.ch}}.}
\and
Benny Sudakov\thanks{Department of Mathematics, ETH, Z\"urich, Switzerland. Email:
\href{mailto:benjamin.sudakov@math.ethz.ch} {\nolinkurl{benjamin.sudakov@math.ethz.ch}}.
Research supported in part by SNSF grant 200021-175573.}
}

\begin{document}

\maketitle
\vspace{-0.8 cm}
\begin{abstract}
    We study the following question raised by Erd\H{o}s and Hajnal in the early 90's. Over all $n$-vertex graphs $G$ what is the smallest possible value of $m$ for which any $m$ vertices of $G$ contain both a clique and an independent set of size $\log n$? We construct examples showing that $m$ is at most $2^{2^{(\log\log n)^{1/2+o(1)}}}$ obtaining a twofold sub-polynomial improvement over the upper bound of about $\sqrt{n}$ coming from the natural guess, the random graph. Our (probabilistic) construction gives rise to new examples of Ramsey graphs, which while having no very large homogenous subsets contain both cliques and independent sets of size $\log n$ in any small subset of vertices. This is very far from being true in random graphs. Our proofs are based on an interplay between taking lexicographic products and using randomness.

\end{abstract}

\section{Introduction}


Ramsey theory refers to a large body of deep results, which roughly say that any sufficiently large structure is guaranteed to have a large well-organised substructure. Its inception dates back to 1929 and the celebrated theorem of Ramsey \cite{ramsey1929problem} which states that any sufficiently large graph must contain a clique or an independent set of arbitrarily large size. In terms of quantitative results in 1935 Erd\H{o}s and Szekeres \cite{E-S} showed that any graph on $n$ vertices contains a clique or an independent set of size $0.5\log n$. On the other hand in what was one of the first applications of the now indispensable probabilistic method Erd\H{o}s \cite{ramsey-random} has shown that in a random graph $\G(n,1/2)$ w.h.p. there are no cliques or independent sets of size $2 \log n.$ Despite considerable effort \cite{explicit-ramsey-1,explicit-ramsey-2,explicit-ramsey-3,frankl-wilson} there are still no known non-probabilistic constructions which match the random graph.

We say a graph is \textit{$k$-Ramsey} if it contains neither a clique nor an independent set of size $k$. In general an $n$-vertex graph is said to be a \textit{Ramsey graph} if it is $k$-Ramsey for some $k$ ``close'' to $\log n.$ Over the years there has been a wide body of work studying properties of Ramsey graphs. In particular, based on the apparent difficulty of finding non-probabilistic Ramsey graphs, it is widely believed that with an appropriate definition of ``close'' any Ramsey graph must be random-like. While there is a vast number of results (see \cite{random-ramsey-property-1,random-ramsey-property-2,random-ramsey-property-3,random-ramsey-property-4,random-ramsey-property-5,random-ramsey-property-6,random-ramsey-property-7, random-ramsey-property-8} and references within) showing that indeed Ramsey graphs need to satisfy, to an extent, various properties usually associated with random graphs, our understanding of Ramsey graphs is still far from sufficient to consider this claim in any way settled.

Given an integer $k$ let $G$ be a $k$-Ramsey graph with the largest number of vertices. Observe that $G$ must contain \textit{both} a clique and an independent set of size $k-1,$ as otherwise we can add a new vertex joined to all or none of the vertices of $G$ to find a larger $k$-Ramsey graph. This shows that if we have a good Ramsey graph the largest clique and largest independent set should be of similar size. For example, if we consider the random graph, which is the best known Ramsey graph, it will with high probability contain both a clique and an independent set of size a $1$ or $2$ less than the largest $k$ for which it is $k$-Ramsey. Furthermore, Ramsey graphs satisfy a similar property locally as well. Given a $k$-Ramsey graph, since it has no clique or independent set of size $k$, we know that any subset consisting of $R(k,\ell)$\footnote{$R(k,\ell)$ denotes the off-diagonal Ramsey number, defined as the minimum number of vertices in a graph needed to guarantee there is either a clique of order $k$ or an independent set of order $\ell$.} vertices contains \textit{both} a clique and an independent set of size $\ell$.

The so called local-global principle, stating that one can obtain global understanding of a structure from having a good understanding of its local properties, or vice versa, has been ubiquitous in many areas of mathematics and beyond for many years \cite{local-global-1, local-global-2, local-global-3, local-global-4, local-global-5}. Keeping this in mind the following problem of Erd\H{o}s and Hajnal \cite{erdos-problems} seems to be very relevant to understanding Ramsey graphs. Given a $k$ (which might be a function of $n$) and an $n$-vertex graph $G$ they ask what is the smallest $m$ for which any $m$ vertex subset of $G$ contains both a clique and an independent set of size $k$? We denote the answer by $m_G(k)$ and say that a graph is \textit{\lramsey{m}{k}} if $m \ge m_G(k)$. To see the relation to Ramsey graphs first observe that being \lramsey{m}{2} is equivalent to being $m$-Ramsey. Secondly, more interestingly if we can find an $n\ge 3m$ vertex \lramsey{m}{k+1} graph it can contain at most $k-1$ vertex disjoint cliques of size $m/k,$ as otherwise they would give us a set of $m$ vertices in which there is no independent set of size $k+1$. The same clearly applies for independent sets. So if we remove a maximal collection of such cliques and independent sets we are left with a graph on at least $n/3$ vertices which is $m/k$-Ramsey in addition to still being \lramsey{m}{k+1}. 

This means that understanding the behaviour of $m_G(k)$ very well could lead us to better understanding of Ramsey graphs, as well as interesting new examples of Ramsey graphs. Let us first consider what happens with $m_G(k)$ for the random graph $G\sim \G(n,1/2).$ If $k$ is small compared to $n$ we have that w.h.p.\ $m_G(k)= \Theta(k \log n)$ (see Section \ref{sec:small-k}) which as we will see, and as one might expect since for $k=2$ this is the standard Ramsey problem, is actually smallest possible among all graphs. On the other hand we also have by Erd\H{o}s's results \cite{ramsey-random} that $m_G(k)\ge 2^{k/2}$ so as $k$ becomes larger than $\log \log n$ the bound deteriorates quickly. For example, one needs at least $\sqrt{n}$ size sets to guarantee to be able to find both cliques and independent sets of size $\log n$. 

A natural question is whether one can do better.
In fact, Erd\H{o}s \cite{erdos-problems} singled out the case of $k=\log n$ and asked if such graphs exist with $m_G(k)=(\log n)^3.$ If the answer were positive this would give rise to $(\log n)^2$-Ramsey graphs which are very different than $\G(n,1/2)$, since they would still be \lramsey{(\log n)^3}{\log n} which is very far from being true in the random graph. This question remains open. However, the first and the third author \cite{local-independence-numbers} show that $m_G(\log n) \ge \Omega ((\log n )^3/\log \log n),$ which perhaps validates Erd\H{o}s' intuition behind asking the question with the parameters he chose. The authors in fact essentially resolve an analogous local Tur\'an type problem which they use to obtain the above-mentioned bound. This problem, for various choices of parameters was also studied in \cite{erdos-problems, LIN-linial, LIN-us, LIN-kriv, LIN-kost}.

On the other hand, in terms of upper bounds nothing better than the one mentioned above, coming from random graphs, namely that there is a graph $G$ for which $m_G(\log n) \le O(\sqrt{n}),$ was known, leaving the possibility that no significantly different Ramsey graphs arise this way. Perhaps surprisingly our main result shows this is not the case, giving a twofold sub-polynomial improvement over the above bound.
\begin{thm}\label{thm:intro-log-n}
There exists an $n$-vertex graph $G$ for which $$m_G(\log n) \le 2^{2^{(\log \log n)^{1/2+o(1)}}}.$$ 
\end{thm}

As discussed above this gives rise to Ramsey graphs which, while being worse than the random graph, are significantly better than the classical explicit construction of Frankl and Wilson \cite{frankl-wilson} and even the recent breakthrough explicit construction of Barak, Rao, Shaltiel and Wigderson \cite{explicit-ramsey-1}. While unfortunately our construction does use randomness, it still gives rise to somewhat weaker Ramsey graphs which are very different from $\G(n,1/2)$ in the sense that they are \lramsey{2^{2^{(\log \log n)^{1/2+o(1)}}}}{\log n}.

So far we have restricted attention to the case of $k=\log n$ for simplicity and to allow for easier comparison between results. We do find examples in the general case as well.

\begin{restatable}{thm}{maingeneral}
\label{thm:intro-main-general}
For any $n \ge 4$ and $k \ge \log n$ there exists an $n$-vertex graph $G$ with $$\log \log m_G(k) \le 6\sqrt{ \log \log n \log \log k}.$$
\end{restatable}
Finally, we prove a simple proposition which determines $m_n(k)$, defined as the minimum of $m_G(k)$ over all $n$ vertex graphs $G,$ up to a constant factor, provided $k$ is small enough compared to $n$.

\begin{prop}\label{prop:small-r}
Provided $n$ is sufficiently large compared to $k\ge 2$ we have $m_n(k)=\Theta(k \log n).$
\end{prop}

\textbf{Notation.}
We denote by $K_k$ the complete graph on $k$ vertices and by $I_k$ the independent set of $k$ vertices. We denote by $\omega(G)$ the clique number of $G$. All our logarithms are in base $2$. When saying a graph $G$ is \lramsey{m}{r} we do not require either $m$ or $r$ to be integers, we want that any set of at least $m$ vertices contains a clique and an independent set of size at least $r.$ 

\section{Locally Ramsey graphs and lexicographic products}\label{sec:prelim}
We begin with a proposition which provides us with a starting point for our further constructions.

\begin{prop}\label{prop:random-graph}
For any $n$ there exists an $n$-vertex graph which is \lramsey{2^{r+8}\log n}{r} for all $r$.
\end{prop}
\begin{proof}
Let us first fix an $r\ge 2$ and set $m=\ceil{2^{r+8}\log n}$. 
The chance that an $m$ vertex induced subgraph of $\G(n,1/2)$ does not contain $K_r$ (or $I_r$) is equal to the chance that $\omega(\G(m,1/2))<r.$ Using Janson's inequality (as in Section 10.3 in \cite{alon-spencer}) implies that this probability is at most $e^{-\frac{\mu^2}{2(\mu+\Delta)}}$ where $\mu=\binom{m}{r}2^{-\binom{r}{2}}$ and $\Delta=\mu^2\cdot \sum_{i=2}^{r-1} \binom{r}{i}\binom{m-r}{r-i}2^{\binom{i}{2}}/\binom{m}{r}.$ In our case $\Delta \le \mu^2 \binom{m}{r}^{-1}\binom{r}{2}\binom{m-r}{r-2}\cdot 2\sum_{i=2}^{r-1}2^{2-i}\le 2\mu^2r^4/m^2.$ It is easy to check that $\mu > m^2/r^4$ and hence $\mu +\Delta \le 2\Delta.$ Thus, $\frac{\mu^2}{2(\mu+\Delta)} \ge \frac{m^2}{8r^4}$ and therefore $\P(\omega(\G(m,1/2))<r) \le e^{-\frac{m^2}{8r^4}}.$ Finally, by a union bound, the probability that there exists a set of $m$ vertices in $\G(m,1/2)$ which does not have $K_r$ or $I_r$ is at most 
$$ \binom{n}{m} \cdot 2 e^{-\frac{m^2}{8r^4}}\le 2e^{m(\log n-m/(8r^4))}\le 2e^{-m\log n/8} \le \frac{1}{n^2}.$$
Here in the second to last inequality we used $2^{r+8}\ge 9r^4.$ Now taking a union bound over all $r \le \log n$ we deduce that the desired graph exists.
\end{proof}

Note that in fact we proved that $\G(n,1/2)$ is \lramsey{2^{r+8}\log n}{r} with high probability. This bound can be slightly improved (see Section \ref{sec:small-k}) but we would gain little in our applications since $m_G(r)$ is going to ``go within a log'' so we chose for simplicity to show the above bound.

As already mentioned in the introduction, the random graph performs close to best possible when $r$ is very small. Our next construction already does much better when $r\ge \log n$, it serves as a basis and an illustration for our main construction presented in the following section. 

\begin{lem}\label{lem:illustration}
For any integer $N\ge 4$ there exists an $N$-vertex \lramsey{m}{r} graph for any $m,r$ which satisfy $\log r \le \frac{(\log m)^2}{2^{9}\log N}$.
\end{lem}

Upon inverting we obtain a graph $G$ for which $m_G(r)\le 2^{16\sqrt{2\log N \log r}}.$ In particular, when $r=\log N$ this is already significantly better compared to about $\sqrt{N}$ in case of the random graph.

The example we use to prove the lemma is the lexicographic product of a random graph $\G(n,1/2)$ with itself multiple times. The lexicographic product $G\times H$ of two graphs $G$ and $H$ is defined as the graph on the vertex set $V(G)\times V(H)$ in which two vertices $(v,u)$ and $(x,y)$ are adjacent iff $v \sim_G x$ or $v=x$ and $u \sim_H y$. We write $G^{\ell}$ for the lexicographic product of $G$ with itself $\ell$ times. The main property of the lexicographic product which makes them natural candidates for our graphs is that clique and independence numbers are multiplicative (see \cite{lex-product-property}). Let us give some intuition as to why this is useful. Let $G \sim \G(n,1/2)$ and let us compare $G^\ell$ with the random graph on the same number of vertices $G' \sim \G(n^\ell,1/2).$ If we take an induced subgraph $H$ of $G$ on $m$ vertices then $H^{\ell}$ gives us a subset of $m^\ell$ vertices of $G^\ell$ which by the multiplicative property above contains both a clique and an independent set of size about $(2\log m)^\ell.$ On the other hand a subset of $m^\ell$ vertices of $G'$ w.h.p. does not have cliques (or independent sets) of size $2\log (m^\ell) = 2\ell \log m.$ This means that, at least if we restrict our attention to subsets of $G^\ell$ arising in this product fashion, $G^\ell$ is a much better \lramsey{m}{r} graph than $G'$ for most choices of the parameters. Of course one may not just restrict attention to such sets. The following lemma allows one to show that even in arbitrary subsets of $G^{\ell}$ it is possible to find big cliques (and independent sets). 

The statement of the next lemma is somewhat technical, one of the reasons for this is that we want to state it in a very general form since we want to use it twice with very different choices of parameters. Second reason is that we believe it might be useful in improving other constructions people might come-up with in the future, as well as possibly for other problems involving subgraphs of lexicographic products. 

Let us sketch the proof idea. The lemma starts with a graph $G$ on $n$ vertices in which for some $2\le r_2<r_3< \ldots <r_k$ we know that any $m_t$ vertices contain both $K_{r_t}$ and $I_{r_t}$ for all $t \ge 2$. We now take a subset $S$ of $G^{\ell}$ in which we want to find a big clique (the argument for an independent set will be analogous). For every vertex $v\in G$ we denote by $S_v$ the subset of $S$ consisting of all elements having $v$ as their first coordinate. We then look at $m_t$ vertices with highest $|S_v|.$  For some $t$ all these vertices need to actually have a reasonably large $|S_v|,$ say at least $m',$ as we know that $\sum_{v \in G} |S_v|=|S|.$ We now use the information that any $m_t$ vertices in $G$ have a clique of size $r_t,$ so in particular among our top $m_t$ vertices some $r_t$ make a clique, say $1,\ldots, r_t$. Now for any two elements of $S$ if their first coordinates are adjacent in $G$ then they are also adjacent in $G^\ell$. So if we look at the sets $S_{i}, i \le r_t$ all the edges between $S_{i}$ and $S_j$ for $i \neq j$ exist. In particular, if we find a clique in each of $S_i$ we may take a union of these cliques to obtain a clique in $S$. Since all vertices in $S_i$ share the first vertex and $|S_i| \ge m'$, finding a clique reduces to finding a clique in a subset of size $m'$ of $G^{\ell-1}$ for which we may use induction. For an example, if we work with the assumption that $G$ is \lramsey{2\log n}{2}, which we can get from the random graph, then $r_2=2, m_2=2\log n$ and let $|S|=m.$ We split in $2$ cases, either some vertex has $|S_v| \ge m/(4\log n)$ or there are $2\log n$ vertices which all have $|S_v| \ge m/(2n)$, since otherwise $|S|=\sum_{v \in G} |S_v|<(2 \log n)\cdot m/(4\log n)+n \cdot m/(2n)=m$. In the former case we take a vertex $v$ with $|S_v|\ge m/(4 \log n)$ and look for a clique in $S_v$. This reduces the task to looking for a clique in a subset of size $m/(4\log n)$ of $G^{\ell-1}$ which we do by induction. In the latter case by our assumption on $G$ among $2\log n$ vertices there must exist an edge $vu$ of $G.$ This means that we can find a clique of twice the size we are guaranteed in a subset of size $m/(2n)$ of $G^{\ell-1},$ which we once again do by induction. Optimising the choice of parameters will already bring us close to the bound in \Cref{lem:illustration}. 

It will be more convenient to work with the inverse of $m_{G^\ell}(r).$ So, let $\beta_G(m,\ell)$ denote the largest $r$ such that in any $m$-vertex subset of $G^{\ell}$ we can always find both $K_r$ and $I_r$.

\begin{lem}\label{lem:lex-product}
Let $G$ be an $n$-vertex graph. Suppose that for some $2 \le r_2 < \ldots <r_k$ we know that $ m_G(r_t)\le m_t,$ for $2\le t\le k$. Then
\begin{equation}\label{eq:beta}
\log \beta_G(m,\ell) \ge \frac{\log m-\ell \log(2m_2)}{\max\limits_{2 \le t \le k} \left( \frac{\log (m_{t+1}/m_2)+t}{\log r_t}\right)},
\end{equation}
for any choice of $m_{k+1}\ge \min(n,m)+1.$
\end{lem}

\begin{proof}
First note that if for some $i<j$ we have $m_i>m_j$ we may decrease $m_i$ to be equal to $m_j$, since $r_i<r_j$ implies $m_j \ge m_G(r_j) \ge m_G(r_i)$ and doing this can only increase the target function. So we may assume that $m_i$ is increasing.

We will prove the claim by induction on $\ell$ for every $m$ satisfying $\min(n,m)+1 \le m_{k+1}$ (where we are treating $G,r_i$'s and $m_i$'s as fixed parameters). 
For the base case of $\ell=1$ since $G^{\ell}=G$ we have $m \le n$ so $m_{k+1}\ge m+1$. We also have $m\ge 2m_2$ as otherwise $\beta(m,\ell)<0$ and the claim is trivial. So there exists some $2\le t\le k$ such that $m_t \le m < m_{t+1}.$ Since $m_t \ge m_G(r_t)$ it is sufficient to show that $\beta_G(m,1)\le r_t$. This indeed holds since $\log m - \log (2m_2)  < \log m_{t+1}-\log m_2= \log (m_{t+1}/m_2).$ 

Now assume $\ell \ge 2$ and that the claim holds for $\ell-1$ and any $m$ for which $\min(n,m)+1\le m_{k+1}.$ Let $S$ be a set of $m$ vertices in $G^{\ell}$. For any vertex $v\in G$ let us denote by $s(v)$ the number of elements in $S$ which have $v$ as their first coordinate. Let us also set $m_1=m_G(1)=1$ and $r_1=1$ for convenience.
\begin{claim*} For some $1 \le t \le k$ there are $m_t$ vertices $v \in G$ with $|S_v| \ge m/(2^t m_{t+1})$.
\end{claim*}

\begin{proof}
Let $v_1, \ldots, v_{s}$ be the vertices of $G$ which have $s(v_i)>0$ ordered so that $s(v_i)$ is decreasing in $i$. Note that $s \le \min(|G|,|S|) \le m_{k+1}-1.$ If $s(v_{m_i}) 
\ge  m/(2^i m_{i+1})$ for some $i\le k$ we may take $t=i$ and $\{v_1,\ldots, v_{m_t}\}$ all have $s(v_i) \ge  m/(2^t m_{t+1})$ so we are done. Hence, we may assume that $s(v_{m_i}) <  m/(2^i m_{i+1})$ for all $i\le k$. But this implies that $$|S|=\sum_{i=1}^s s(v_i)=\sum_{i=1}^{k}\sum_{j=m_i}^{m_{i+1}-1}s(v_j) \le \sum_{i=1}^{k} (m_{i+1}-m_i)\cdot m/(2^i m_{i+1})<m\sum_{i=1}^{k} \frac{1}{2^i}<m,$$ which is a contradiction.
\end{proof}

Since $m_t \ge m_G(r_t)$ among the $m_t$ vertices given by the claim there must be $r_t$ forming a $K_{r_t}$ in $G$. Let us denote by $S_1,\ldots,S_{r_t}\subseteq S$ the sets of elements in $S$ whose first coordinate is the $i$-th vertex of this clique. 
By definition of the lexicographic product if $2$ elements in $S$ have adjacent first coordinates they are adjacent in $G^\ell$ as well. This means that we can combine the cliques we find in each of $S_i$ into a single clique in $S$. Furthermore, since all elements in $S_i$ have the same first coordinate we can delete their first coordinate when looking for a clique, which leaves us with a subset of at least $m'=m/(2^{t} m_{t+1})$ elements of $G^{\ell-1}$ to which we can apply induction\footnote{Note that $m'<m$ so $\min(n,m')\le \min (n,m)\le m_{k+1}-1$ as required by the assumption.} to find a clique. This means we can find a clique in $G^{\ell}$ of size at least $r_t \cdot \beta_G(m',\ell-1)$. Repeating the argument for independent sets we can find an independent set of this size as well. Therefore:
\begin{align*}
    \log \beta(m,\ell) &\ge \log \left (r_t \cdot \beta\left(m',\ell-1\right)\right) = \log r_t + \frac{\log m' -(\ell-1)\log (2m_2)}{C} \\
    & = \frac{C\log r_t+\log m-\log (2^{t} m_{t+1})-\ell \log (2m_2)+\log (2m_2)}{C} \\
    & \ge \frac{\log m-\ell \log (2m_2)}{C}
\end{align*}
where $C=C(m_2,\ldots,m_{k+1},r_2,\ldots,r_k)=\max\limits_{2 \le t \le k} \left( \frac{\log (m_{t+1}/m_2)+t}{\log r_t}\right)$ and the last inequality follows since we get an equality if $t=1$ and since $C \log r_t \ge \log m_{t+1}-\log m_2+t$ if $t \ge 2.$ As this is precisely the RHS of \eqref{eq:beta} this completes the proof.
\end{proof}

\Cref{lem:illustration} now follows as a corollary upon making the appropriate choice for the parameters.

\begin{proof}[ of \Cref{lem:illustration}]
Let $G$ be the $n$-vertex graph which is \lramsey{2^{t+8}\log n}{t} for all $t$ provided by \Cref{prop:random-graph}. This implies we can apply \Cref{lem:lex-product} with $r_t=t$ and $m_t:=2^{t+8}\log n \ge m_G(t)$ for $2 \le t \le k \le \log n,$ where $k$ is the largest integer such that $m_{k}\le n,$ and $m_{k+1}:=2^{k+9}\log n>n$. The lemma implies that
\begin{equation}\label{eq:1}
    \log \beta_G(m,\ell) \ge  \frac{\log m-\ell\log \log n-11\ell}{2\log n/\log \log n},
\end{equation}
since $\log m_2=10+\log \log n,$ $\log (m_{t+1}/m_2)=t-1$ and $(2t-1)/\log t$ is increasing in $t.$

We claim that $G^\ell$ with an appropriate choice of parameters $n$ and $\ell$ (in terms of $N$ and $m$) provides us with the desired graph. Note that we may assume that $m \le N$ (as otherwise the claim is vacuous) and that $\log m \ge 2^4 \sqrt{ 2\log N}\ge 2^4\log \log N$ as otherwise the claimed bound holds trivially. Let $n$ be the smallest integer for which $\log n / \log \log n \ge 32\log N /\log m$ and let $\ell:= \ceil{\log m/(32 \log \log n)}.$  With this choice of $n$ and $\ell$ we get a graph on $n^\ell \ge N$ vertices, since $$\ell \log n\ge \frac{\log m}{32 \log \log n} \cdot \frac{32\log N \log \log n}{\log m} =\log N.$$ Furthermore, by \eqref{eq:1} any set of size $m$ contains both cliques and independent sets of size at least $2$ to the power
$$\frac{\log m - 12 \ell \log \log n}{2\log n/ \log \log n} \ge \frac{\log m - \frac{12}{16}\log m }{2\cdot 33\log N /\log m} = \frac{(\log m)^2}{33\cdot 8\log N}$$
where in the denominator we used that $\log n / \log \log n < 33\log N /\log m$ (which holds since $\log n /\log \log n$ grows slower than $n$ and $\log N/\log m \ge 1$) and in the numerator $\ell \le \log m/(16 \log \log n)$ (which holds since $\log m/(32 \log \log n)\ge 1/2$ and $\ceil{x} \le 2x$ for any $x \ge 1/2$).
\end{proof}

\section{Locally Ramsey graphs and scrambling}
It might be tempting to try to reiterate the argument used in the previous section by starting with our better construction in place of the random graph. Notice however, that all our examples are in fact already powers of the random graph so doing this would only provide us with higher powers of the random graph which are already considered by our argument. This idea however has some merit when combined with a further twist. If we start with a high power of the random graph it will be a much better \lramsey{m}{r} than the random graph for some fixed value of $r$ but perform comparatively poorly for small values of $r.$ If we now scramble this graph a little bit, in the sense that we flip every edge and non-edge with some small probability this will improve the performance of our graph when $r$ is small while only slightly decreasing performance for larger $r$. Taking the lexicographic powers of this graph in place of the random graph is how we obtain our improved construction.

Let us define the $p$-scramble $\G(G,p)$ of a graph $G$ to be the graph obtained by independently removing every edge of $G$ and adding every non-edge of $G$ with probability $p$. The following lemma makes formal the above idea that by taking a $p$-scramble of an \lramsey{m}{r} graph $G$ we obtain a graph which is close to being as good a Ramsey graph as $\G(n,p),$ meaning it has no cliques or independent sets of size about $\log n/p$ or in other words is \lramsey{\log n/p}{2} but is in addition still close to being \lramsey{m}{r}.

\begin{lem}\label{lem:randomize}
If there exists an $n$-vertex \lramsey{m}{r} graph then there exists a graph which is both \lramsey{m}{\frac{r}{17 \log n}} and \lramsey{r/2}{2}, provided $r \ge 16 \log n$. 
\end{lem}

\begin{proof}

Let $G$ be an \lramsey{m}{r} graph on $n$ vertices. Let $G'\sim \G(G,p)$ with $p:=\frac{8\log n}{r}\le 1/2.$ 

By an immediate coupling, the probability that $G'$ contains a $K_{r/2}$ (or $I_{r/2}$) is at most the probability that $\G(n,1-p)$ contains a $K_{r/2}.$ This probability is, by a union bound, at most $\binom{n}{r/2}(1-p)^{\binom{r/2}{2}}\le 2^{r/2 \cdot \log n-pr^2/8}=n^{-r/2} \le \frac14$ since $\frac r2=\frac{4\log n}{p}$. So $G'$ contains neither $K_{r/2}$ nor $I_{r/2},$ or in other words is \lramsey{r/2}{2} with probability more than $1/2$. 

Given a set of $r$ vertices forming a clique in $G$ the probability that in $G'$ this set  still contains a clique of size $t$ is equal to the probability that $\G(r,1-p)$ has a $K_t$. Note that the expected number of missing edges is $\mu=\binom{r}{2}p=4(r-1)\log n$ so by Chernoff's inequality (see Appendix A of \cite{alon-spencer}) the probability that there are more than $2\mu$ edges is at most $e^{-\mu/3}$. If we have less edges then by Turan's theorem (see \cite{alon-spencer}) there is a clique of size at least $\frac{r^2}{4\mu+r}\ge \frac{r}{16r\log n+r} \ge \frac{r}{17 \log n}$. This means that with probability $1- \binom{n}{r}e^{-\mu/3}\ge 1-n^{r-\frac{4}{3}(r-1)}\ge 1/4$ any clique of size $r$ in $G$ contains a clique of size at least $\frac{r}{17 \log n}$ in $G'$. Repeating for the independent sets we conclude that $G'$ is \lramsey{m}{\frac{r}{17 \log n}} with probability at least $1/2$. Therefore, with positive probability the desired graph exists.
\end{proof}

Being more careful one can improve $\frac{r}{17 \log n}$ to $\Omega\left(\frac{r\log r}{\log n}\right)$ but not more (since $\omega(\G(r,1-p))=\Theta\left(\frac{\log r}{p}\right)$ w.h.p.) However, this improvement seems to be negligible in our applications so we opted for the above simpler argument. The following lemma gives our main construction. We obtain it by starting with our construction from the previous section, scrambling it using \Cref{lem:randomize} then taking an appropriate lexicographic power using \Cref{lem:lex-product} and repeating with this new graph. The parameter $t$ will control the number of iterations that we do. We also, for now, add an assumption that the clique/independent set size $r$ we are looking for is not too small. 

\begin{thm}\label{thm:main-general}
For any $t \ge 2$ there exists a \lramsey{m}{r} graph on $N\ge 4$ vertices, provided $\log m \ge 
t^{2t}(\log r)^{t} (\log N)^{1/t}$ and $\log r \ge t\log \log N$.
\end{thm}

\begin{proof}
Let us define $\log m(N,r,t):=t^{2t}(\log r)^{t} (\log N)^{1/t}.$ We will prove by induction on $t$ that for any $N \ge 4$ and $\log r\ge t \log \log N$ there exists an \lramsey{m(N,r,t)}{r} graph on $N$ vertices. The base case of induction for $t=2$ follows (with room to spare) from \Cref{lem:illustration}.

Let us take an \lramsey{m'}{r'} graph $G$ on $n \ge 4$ vertices 
(with parameters $r',n$ satisfying $r' \ge 16 \log n,$ $\log r' \ge t \log \log n$ and $n \ge 4$, to be chosen later) 
given by the inductive assumption for some $t\ge 2$, so with $m':=m(n,r',t)$. Let $G'$ be the scrambled graph given by \Cref{lem:randomize} applied to $G.$ So in particular $G'$ is \lramsey{r'/2}{2} and \lramsey{m'}{r''} where we write $r'':=\ceil{{r'}/{(17 \log n)}},$ note that we are using $r' \ge 16 \log n$ so that the lemma applies.

We now take lexicographic products of $G'.$ \Cref{lem:lex-product} allows us to use these locally Ramsey properties of $G'$ to give bounds on the locally Ramsey properties of $G'^\ell$ which will be our actual example for iteration $t$. 
So in particular we may apply \Cref{lem:lex-product} with $r_2=2,m_2=r'/2 \ge m_{G'}(2); r_3=r'', m_3=m' \ge m_{G'}(r'')$ and $m_4=m+1$. The lemma implies  
\begin{equation}\label{eq:2}\log \beta_{G'}(m,\ell) \ge \frac{\log m-\ell \log r'}{\max\left(\log m', \log m / \log r''\right)},
\end{equation}
since $\log (2m'/r')+2\le \log m'$ (since $r' \ge 8$) and $\log (2(m+1)/r')+3 \le m$ (since $r' \ge 16 \log n \ge 32$). 

We note that at this point what remains to be done is to choose the parameters and use \eqref{eq:2} to show the induction step holds. The rest of the proof is somewhat technical and it might help the reader to at first ignore various constants and floors and ceils.
Let us now choose all our parameters in terms of $N,r$ and $t.$ Let $$m:=m(N,r,t+1),\:\:\:\: \log r':=\left(1+\frac{2}{t}\right)\log r, \:\:\:\: \ell:=\floor{\frac{\log m}{(t+1)\log r'}} \:\: \text{ and } \:\: \log n:=\ceil{\frac{\log N}{\ell}}.$$ Our goal is to show that with this choice of parameters the RHS of \eqref{eq:2} is at least $\log r.$ This would give us a graph on $n^\ell \ge N$ (by definition of $n$) vertices which is \lramsey{m}{r} so we obtain the desired graph by taking a subgraph consisting of exactly $N$ vertices. We first show the following easy inequalities.

\begin{claim*}
We have $16 \le n \le N,$ $\ell \ge 64$ and $\log r' \ge (1+1/t)(\log r+\log (17\log n)).$ 
\end{claim*}
\begin{proof}
Note that since $\log m=\log m(N,r,t+1)= (t+1)^{2(t+1)}(\log r)^{t+1} (\log N)^{1/(t+1)}>(t+1)^{2(t+1)}\log r$ so in particular $\frac{\log m}{(t+1)\log r'}= \frac{\log m}{(t+1)(1+2/t)\log r} \ge (t+1)^{2(t+1)-1}/(1+2/t) > 64,$ which in turn implies $\ell=\floor{\frac{\log m}{(t+1)\log r'}} \ge 64.$ This together with $N \ge 4$ and the definition of $n$ imply $n < N.$ If $m>N$ then there are no subsets of size at least $m$ in $G'^\ell$ so the induction step is vacuously true, therefore we may assume $m \le N$. Using this we get $\log n \ge \frac{\log N}{\ell} \ge \frac{\log N \cdot (t+1)\log r'}{\log m} \ge (t+1)\log r'>3,$ (where we used $r'>1$ and $t \ge 2$). This in particular implies that $\log n \ge 4$ and $\log n \le 4/3 \cdot \log N /\ell \le \log N / 32$ (using $\ell \ge 64$). This in turn implies $\log r' = (1+2/t)\log r\ge (1+1/t)\log r+\log \log N \ge (1+1/t)\log r+ \log (17\log n),$ where we are using $1/t \cdot \log r \ge \log \log N \ge \log (32 \log n).$ 
\end{proof}

This immediately implies the required inequalities on $n,r'$ and $t,$ indeed $ r' \ge 17 \log n$ and $n \ge 4$
while $\log r \ge (t+1)\log \log N$ implies $\log r' \ge \log r \ge (t+1)\log \log N \ge t \log n.$

Let us now turn to the main inequalities. Observe that
\begin{equation}\label{eq:3} \frac{\log m- \ell \log r'}{\log m / \log r''} \ge \left(\frac{\log m-\log m/(t+1)}{\log m}\right)\log r'' \ge \left(1-\frac1{t+1}\right)(\log r'-\log (17\log n)) 
\ge \log r,
\end{equation}
where in the last inequality we used the main inequality from the claim. This shows that one of the two desired inequalities that we need to show to conclude that RHS of \eqref{eq:2} is at least $\log r$ holds. The second inequality we need is equivalent to $\log m- \ell \log r' \ge \log m' \log r$ and is implied by $\log m \ge (1+1/t)\log m' \log r$ (since $\ell \le \log m /((t+1)\log r')$). Let us now show this inequality holds (recall that we have chosen $m'=m(n,r',t)$):
\begin{align*}
 (1+1/t)\log r \log m'  & = (1+1/t)\log r \cdot t^{2t}(\log r')^{t} (\log n)^{1/t}\\   
  & \le (t+1)t^{2t-1}\log r \cdot (\log r')^{t+1/t} \cdot 2 \left(\frac{\log N}{\log m}\right)^{1/t}\\
  &\le (t+1)^{2(t+1)} (\log r)^{t+1+1/t} (\log N)^{1/t}\cdot (\log m)^{-1/t}\\   
  & \le (\log m)^{1+1/t}\cdot (\log m)^{-1/t} =\log m,
\end{align*}
where in the first inequality we used $\log n \le \frac{4\log N}{3 \ell}\le \frac{4(t+1)\log r'\log N}{3 \log m}$ (following since $n \ge 16$ so $\log n \ge 4$ and definitions of $n$ and $\ell$) and $(4(t+1)/3)^{1/t}\le 2$ (since $t\ge 2$). In the second inequality we used $\log r' = (1+2/t)\log r$ and $(1+2/t)^{t+1/t}\le (1+1/t)^{2(t+1/t)}\le (1+1/t)^{2t+1}.$ The third inequality follows as $m=m(N,r,t+1)=(t+1)^{2(t+1)}(\log r)^{t+1}(\log N)^{1/(t+1)}$ since $\left(t+1+1/t\right)\cdot t/(t+1)=t+1/(t+1) \le t+1.$
Together with \eqref{eq:3} this shows that the RHS of \eqref{eq:2} is at least $\log r$ completing the proof.
\end{proof}
We were relatively lax with various estimates in the argument above for the sake of simplifying the inequalities as much as possible. For example, the optimal exponent of $\log r$ (which one may obtain using exact same parameters as we did above) is $(t+1)/2-2/t.$ We also note that the assumption $r \ge (\log n)^t$ was also made for the sake of simplicity, since for smaller values of $r$ the above argument would give barely any improvement over just using the above bound with $r=(\log n)^t$ and monotonicity of $m_G(r)$ in $r$. Let us now optimise over $t$ and obtain \Cref{thm:intro-main-general} as a corollary. Recall the statement of \Cref{thm:intro-main-general}.

\maingeneral*
\begin{proof}
If $\log N \le k \le (\log N)^t$ we use \Cref{thm:main-general}, with $r=(\log N)^t$ to obtain a graph $G$ with $\log m_G(k) \le \log m_G((\log N)^t)\le t^{3t}(\log \log N)^t(\log N)^{1/t}\le t^{3t}(\log k)^t(\log N)^{1/t}.$ If $k>(\log N)^t$ we may use \Cref{thm:main-general} directly with $r=k$ to conclude that for any $k \ge \log N$ there is a graph $G$ with $\log m_G(k) \le  t^{3t}(\log k)^t(\log N)^{1/t}.$ 

We now choose $t=\floor{ \sqrt{\frac{\log \log N}{\log \log k}}}$. If $t < 2$ we obtain that the desired inequality requires $m>N$ making the claim vacuous. Hence the above inequality gives us 
$$\log \log m_G(k) \le 3t \log t+ t\log \log k+\frac{1}{t}\cdot \log \log N \le 4t\log \log k+\frac{1}{t}\cdot \log \log N\le 6\sqrt{ \log \log N \log \log k}$$
where we used $t \le \log \log N \le \log k$ and $t\ge 2.$
\end{proof}

\Cref{thm:intro-log-n} follows by simply plugging in $k= \log N$ in \Cref{thm:intro-main-general}.

\textbf{Remark.} By following the argument used in \Cref{prop:random-graph} it is not hard to show that $\G(n,p)$ is w.h.p.\ \lramsey{(1/p)^{r+8}\log n}{r} for all $r$, assuming $p \le 1/2$. Using this in \Cref{lem:randomize} would give us a graph which is \lramsey{(1/p)^{r+8}\log n}{r} for all $r$ in addition to being \lramsey{m}{\frac1{3p}}, provided $1/p\ge r/(8\log n)$. We can then use this extra information for small values of $r$ similarly as we did in the proof of \Cref{lem:illustration} to obtain an improvement in \eqref{eq:2}. Ultimately, this would lead to an improvement in \Cref{thm:main-general} in which we divide by roughly a $(\log \log N)^{(t-3)/2}$ factor which would only slightly improve the $o(1)$ term in \Cref{thm:intro-log-n}.

\section{Small cliques and independent sets}\label{sec:small-k}
In this section we show \Cref{prop:small-r}. We begin with the lower bound.
\begin{prop}\label{prop-lb}
Provided $n\ge 4r \log n$ and $r\ge 2$ we have $m_n(r)\ge (0.5+o(1))r \log n.$
\end{prop}
\begin{proof}
Let us start with the lower bound. Given an $n$-vertex graph $G,$ by the standard bound on Ramsey numbers (see e.g. \cite{E-S}), $G$ must contain a clique or an independent set of size at least $0.5 \log n$. If we remove this set and repeat $2r - 3$ times we get either $r - 1$ vertex disjoint cliques or $r - 1$ vertex disjoint independent sets of size at least $0.5  \log (n/2)$ as at each step we are left with at least $n-2r\log n \ge n/2$ vertices. The union of these sets give us a set of $(0.5+o(1))r \log n$ vertices in which we can not find a clique (if the sets were independent) or an independent set (if the sets were cliques) of size $r.$ This shows that $m_G(r)\ge (0.5+o(1))r \log n.$
\end{proof}

\textbf{Remark.} The above bound applies for essentially the whole range but is beaten by the approach in \cite{local-independence-numbers} as soon as $r$ is bigger $\log \log n$. They show that provided there is an $I_r$ in every subset of size $s$ then one can find an independent set of size $\Omega(r \log (n/s) /\log (s/r))$. Using this for our graphs and finding $r$ copies of this big independent set as in the proof above would show $m_n(r) \ge \Omega(r^2 \log n /\log \log n),$ provided $r$ is at most polylogarithmic in $n$ which beats the bound in \Cref{prop-lb} when $r\gg \log \log n$.

Let us now turn to the upper bound. Perhaps not too surprisingly since we are working with ``small'' values of $r$ the example is the random graph $G \sim \G(n,1/2)$. Since $n$ is much bigger than $r$ the argument we used to prove \Cref{prop:random-graph}, even when done more carefully, would only give us $m_G(r) \le O(r^4 \log n)$ so we make use of a slightly different approach. 

\begin{prop}
Provided $n$ is sufficiently large compared to $r$ we have for $G\sim \G(n,1/2)$ that w.h.p.\ $m_G(r)= \Theta(r \log n).$
\end{prop}

\begin{proof}
The lower bound follows from the previous proposition. So we focus on the upper bound.
It is well-known (see \cite{chi-clique}) that almost all graphs without $K_r$ are
$r - 1$ colourable. Hence, provided $m \to \infty $ as $n \to \infty$ we have $\P(\omega(\G(m,1/2))<r)\le (1+o(1))\P(\chi(\G(m,1/2))<r) \le (1+o(1))r^{m}2^{-(1+o(1))m^2/(2r)}.$ Here, the last inequality follows since there are $(r-1)^m$ many ways to assign $r-1$ colours to $m$ vertices and given a colouring there are at least $(r-1)\binom{m/(r-1)}{2}\ge (1+o(1))m^2/(2r)$ pairs of vertices assigned the same colour which are not allowed to appear as edges.
The same estimate holds for the probability that a graph on $m$ vertices contains no independent set of size $r$. Thus in $G$ the expected number of sets of size $m$ which contain no clique of size $r$ or no independent set of size $r$ is at most
$$\binom{n}{m}\cdot 2(1+o(1))r^{m}2^{-(1+o(1))m^2/(2r)}$$
which tends to $0$ for $m = (2 + o(1))r \log n,$ completing the proof.
\end{proof}
The fact that almost all $m$-vertex $K_r$-free graphs are $r-1$ colourable has recently been shown to be true for $r$ up to $\log m/(10 \log \log m)$ in \cite{containers}. Since our sets have size $m$ which is roughly $r \log n$ this means that $n$ sufficiently large in the above result may be replaced with $r \le O(\log \log n/ \log \log \log n).$

\section{Concluding remarks}
In this paper we study the function $m_G(r)$ with particular interest in how small it can be. The function $m_n(r)$ defined as the minimum of $m_G(r)$ over all $n$-vertex graphs $G$ was introduced by Erd\H{o}s and Hajnal almost 30 years ago. Combined with the lower bound obtained in \cite{local-independence-numbers} we obtain
$$\frac{(\log n)^3}{\log \log n}\le m_n(\log n)\le 2^{2^{(\log \log n)^{1/2+o(1)}}}.$$

In general it would be very interesting to get better bounds on $m_n(\log n)$ and in particular answer Erd\H{o}s' question of whether $m_n(\log n)> (\log n)^3$. In fact, the authors suspect that $m_n(\log n)$ may be bigger than any fixed power of $\log n$. Our initial examples in Section \ref{sec:prelim} are essentially the classical examples of explicit Ramsey graphs due to Naor \cite{Naor-lex}, following-up on the idea of using lexicographic products to build Ramsey graphs due to Abbott \cite{abbott-explicit-ramsey}. There has recently been some major progress on finding better explicit Ramsey examples \cite{explicit-ramsey-1,explicit-ramsey-2,explicit-ramsey-3}. It would be interesting to see if one can combine these graphs with our ideas to improve our upper bound.


Another possibly interesting perspective arises if we consider a colouring restatement of our problem. Note that $m_n(r)-1$ may be defined as the largest number $m$ such that in any $2$-colouring of $K_n$ we can find $m$-vertices not containing a monochromatic $K_r$ in one of the colours. With this in mind one can define the $m$-local Ramsey number $LR_m(G)$ of a graph $G$ as the smallest $n$ for which in any $2$-colouring of $K_n$ there are $m$ vertices not containing a monochromatic copy of $G$ in one of the colours. For example if $m$ is sufficiently larger than $r$ \Cref{prop:small-r} implies that $LR_m(K_r)=2^{\Theta(m/r)}$. Natural generalisations to more colours or asymmetric graphs might hold some interest as well.

\vspace{0.2cm}
\textbf{Acknowledgements.} We would like to thank David Conlon for useful conversations and remarks and anonymous referees for their comments. 
\providecommand{\bysame}{\leavevmode\hbox to3em{\hrulefill}\thinspace}
\providecommand{\href}[2]{#2}

\end{document}